\documentclass[oneside,11pt,reqno]{amsart}
\usepackage{amssymb,amsmath,amsthm,bbm,enumerate,mdwlist,url,multirow,hyperref,amsthm}
\usepackage[pdftex]{graphicx}
\usepackage[shortlabels]{enumitem}

\theoremstyle{definition}
\newtheorem{definition}{Definition}
\theoremstyle{theorem}
\newtheorem{proposition}[definition]{Proposition}
\newtheorem{lemma}[definition]{Lemma}
\newtheorem{theorem}[definition]{Theorem}

\newtheorem{assumption}[definition]{Assumption}
\numberwithin{equation}{section}
\theoremstyle{remark}
\newtheorem{remark}[definition]{Remark}
\newtheorem{example}[definition]{Example}
\def\PP{\mathsf P}

\def\RR{\mathbb R}
\def\ZZ{\mathbb Z}

\def\pr{\mathrm{pr}}
\def\supp{\mathrm{supp}}
\def\id{\mathrm{id}}
\def\IM{\mathrm{Im}}

\def\CC{\mathcal C}
\def\DD{\mathcal D}
\def\FF{\mathcal F}
\def\LLL{\mathcal L}

\def\LL{\pmb{\lor}}
\def\ox{\overline{x}}
\def\ux{\underline{x}}
\begin{document}
\title{Copulas for maxmin systems}

\author{Matija Vidmar}
\address{Department of Mathematics, University of Ljubljana and Institute of Mathematics, Physics and Mechanics, Slovenia}
\email{matija.vidmar@fmf.uni-lj.si}

\author{Matja\v{z} Omladi\v{c}}
\address{Institute of Mathematics, Physics and Mechanics, Ljubljana, Slovenia}
\email{matjaz@omladic.net}

\begin{abstract}
Under a mild condition we give closed-form expressions for copulas of systems that consist of maxima and of minima of subvectors of a given random vector $X$ with continuous marginals. Said expressions appear explicit in the copula of $X$ and the mentioned condition is for example met when the law of $X$ admits a strictly positive density with respect to Lebesgue measure. 
In the i.i.d. case these ``maxmin'' copulae become universal and the conditions on their validity can be dropped entirely. Our main motivation comes from applications to shock models that arise in multivariate survival theory, and indeed the maxmin copulas presented herein are connected to/extend the Marshall-Olkin copulas \cite{MO67} going through to the copulas given in \cite{OmlRuz15}. Another application is to order statistics copulas. 
\end{abstract}

\keywords{Copulae; max and min systems; transformations of distribution functions; POD property; survival analysis; shock models}
\subjclass[2010]{Primary: 	60E05; Secondary: 60E15, 62N05}

\maketitle

\section{ Introduction}
As already mentioned in the abstract, our main motivation are shock models as described in \cite{MO67,OmlRuz15,DurOmlOraRuz15}, see also the citations given there. These models are part of survival analysis, an area of Statistics, closely related to reliability theory, event history theory, and possibly some other similar theories, that have a substantial theoretical part in common, and whose name may depend on whether they are dealing primarily with problems coming from life sciences, engineering, economics, sociology, or somewhere else. The univariate case of these theories is to do with notions such as survival function (or reliability function), lifetime distribution function, event density, and hazard function, whilst the multivariate case uses most of the known multivariate statistical methods which include copulas -- functions linking univariate distributions into a multivariate one. 

We refer to \cite{nelsen} for the general theory of copulas and to \cite{joe} for multivariate models with emphasis on copulas. Indeed, a natural application of copulas in survival analysis comes through shock models and was probably for the first time studied in \cite{MO67}. The maxmin copulas that we introduce, may be seen as an immense extension of the Marshall-Olkin copulas \cite{MO67}. A consequence of our results on the maxmin copulas is also the closed form order statistic copula which is an extension of the bivariate results from \cite{schmitz,anjos} and a complement to several other results on this theme \cite{jaworski,mendes}.

Here is now a shock model of the kind that our theory may be applied to. A random vector $Y=(Y_i)_{i=1}^n$ is given, representing, say, lifetimes of certain organisms in life sciences or of certain components of a system in engineering, or perhaps times-to-default of certain financial entities in a portfolio assessed under credit risk evaluation. These variables may be seen as idiosyncratic shocks, i.e.\ they represent the individual life times of, say, the components of the system. Furthermore, we have a random vector $Z=(Z_j)_{j=1}^m$ representing external shocks, sometimes also called exogenous or systemic shocks. For any two indices $i,j,1\leqslant i\leqslant n, 1\leqslant j\leqslant m$, we know the effect of the shock $Z_j$ on the $i$-th component. The shock may bring an event that is detrimental to it so that the random variable representing the resulting lifetime equals $\min\{Y_i,Z_j\}$. On the other hand, if this shock brings a beneficial event to the component, the resulting lifetime equals $\max\{Y_i, Z_j\}$. Beneficial shocks may appear in applications, say in engineering, when the component of the system has a recovery option for the given shock. The reader may find more on this including some concrete examples in \cite{OmlRuz15, DurOmlOraRuz15}.

The solution to the problem of the probabilistic dependency structure of the resulting lifetimes, as described in the previous paragraph, brought us to the study of copulas for systems consisting of minima and of maxima of subvectors of a given random vector (maxmin systems). These in turn can be investigated successfully under relatively mild conditions (see Assumption~\ref{assumption}) and, while in general, though not in the i.i.d. case, lacking some measure of aesthetic appeal, appear explicit and computationally tractable; a fact that may be of use to practitioners who are interested in various properties that copulas may or may not have and that are discussed extensively in \cite{nelsen}. We present some flavor of that in Subsections~\ref{subsection:POD} and~\ref{subsection:zero}. The explicit formula itself, called the \textbf{maxmin copula}, is given in relative generality (modulo only Assumption~\ref{assumption}) in \eqref{main} of Subsection~\ref{subsection:main_result} and is our man result, while in Subsection~\ref{subsection:survival} we propose to view this formula as a transformation and introduce the copula, obtained from a given copula, by applying the maxmin transformation $m$ times. We show that these copulas, in principle still obtainable in a closed form, called \textbf{$m$-fold maxmin copulas}, solve the problem of the shock model above.

It is true that our primary motivation lies in these applications. Nevertheless, the explicit solution to this ``maxmin problem'' that we present here, as we believe, may deserve to have been discovered in its own right. Indeed our results, including Lemma~\ref{lemma:key} --- an ancillary claim characterizing when one distribution function can be transformed into another monotonically and continuously, or just monotonically ---, should also be of some independent interest for the development of probabilistic concepts.

Here is how the remainder of our paper is organized. Section~\ref{section:preliminaries} gives the setting and fixes notation. Our main results are then developed in Section~\ref{section:main}, where we first solve the maxmin problem on the level of probabilities  \eqref{first_step}, and reach the final result \eqref{main} after introducing the relevant ``distortion functions''. We also exhibit simplifications of this formula in a couple of special cases. In Section~\ref{section:properties} we present some properties and applications. In particular,  Subsection~\ref{subsection:order} renders our application to order statistics, Subsection~\ref{subsection:survival} comments on the successive formations of maxmin systems and stages an extension of the Marshall-Olkin copulas and the copulas introduced in \cite{OmlRuz15}, finally Subsection~\ref{subsection:low-dimensional} makes explicit some low-dimensional examples. The development of auxiliary tools is deferred to the Appendix, i.e.\ to Section~\ref{section:auxiliary}.

\section{Setting, assumptions, notation}\label{section:preliminaries}

Throughout a probability space $(\Omega,\FF,\PP)$ remains fixed. Let $n\in \mathbb{N}$, and let $X=(X_i)_{i=1}^n$ be a random vector, with values in $\mathbb{R}^n$, $C$ a copula thereof (so that, in the obvious notation, $F_X=C\circ (F_{X_i}\circ \pr_i)_{i=1}^n$). Let also the sets $\CC$ and $\DD$ both consist of non-empty subsets of $[n]:=\{1,\ldots,n\}$, and let them not both be empty, i.e. $\CC\cup \DD\ne \emptyset$ and each element of $\CC\cup \DD$ is a non-empty subset of $[n]$.

Our aim is then to consider the maxmin system $X^{\lor\land}:=((\lor_MX)_{M\in \CC},(\land_MX)_{M\in \DD})$, where for non-empty $M\subset [n]$ we have denoted $\lor_MX:=\max\{X_i:i\in M\}$ and $\land_MX:=\min\{X_i:i\in M\}$.
Specifically, our primary objective is to provide an expression (in terms of $C$ and the marginal distributions $F_i$, $i\in [n]$) for the copula of this $\vert\CC\vert +\vert\DD\vert$-dimensional random vector. We will succeed in doing so under the following (technical) assumption on the relevant distribution functions (we denote for short by $F_i$ the distribution function $F_{X_i}$ of $X_i$, $i\in [n]$):

\begin{assumption}\label{assumption}
For all $i\in [n]$ (1) the $F_i$, $i\in [n]$, are continuous /which implies that $C$ is unique/; (2) for $M$ from $\DD$, respectively $\CC$, whenever $i\in M$, then the intervals of constancy of $F_{\land_MX}$, respectively $F_{\lor_MX}$, are also intervals of constancy of $F_i$.
\end{assumption}

\begin{remark}\leavevmode\label{remark:stable}
\begin{enumerate}[(a)]
\item The situation described in Assumption~\ref{assumption} occurs if the law of $X$ admits a strictly positive density with respect to Lebesgue measure (in which case all of the distribution functions mentioned therein are in fact strictly increasing).
\item\label{remark:stable:b} Property (1) of Assumption~\ref{assumption} is stable under the formation of the maxmin system, i.e. if $X$ consists of continuous random variables, then so does $X^{\lor\land}$ (cf. e.g. the first paragraph of Subsection~\ref{subsection:distortion}). As regards Property (2) of Assumption~\ref{assumption} we note as follows:
\begin{quote}
Suppose (i) the $F_i$, $i\in [n]$, share their intervals of constancy and (ii) for all $a<b$ from $\mathbb{R}$: if $dF_i(a,b]=\PP(X_i\in (a,b])>0$ for some (then all) $i\in [n]$, then $\PP(X\in (a,b]^n)>0$ (this of course obtains if $X$ consists of i.i.d. random variables). Then if the random variable $Y$ is \emph{any} non-empty combination of infima and suprema of the random variables constituting $X$ (in whatever order, and regardless of the placing of the parantheses), i.e. if $Y$ is one of the random variables obtained in the successive formation of the maxmin systems, starting with $X$, then $F_Y$ shares its intervals of constancy with the $F_i$, $i\in [n]$.
\end{quote}
\begin{proof}
For, if $dF_i(a,b]>0$, for some (then all) $i\in [n]$, then by (ii), thanks to $\{Y\in (a,b]\}\supset \cap_{i\in [n]}\{X_i\in (a,b]\}$, and from the monotonicity of probability measures, we obtain $dF_Y(a,b]>0$. Conversely if $dF_Y(a,b ]>0$, then if $dF_i(a,b]=0$ for some (then all) $i\in [n]$, a contradiction would result with $\{Y\in (a,b]\}\subset \cup_{i\in [n]}\{X_i\in (a,b]\}$ via sub-additivity of probability measures.
\end{proof}
Thus if $X$ satisfies (i) and (ii) (in particular, if $X$ is i.i.d.), then all the successive formations of the maxmin systems, starting from $X$, will verify property (2) of Assumption~\ref{assumption} also (cf. Subsection~\ref{subsection:survival}).
\item To see how the provisions of Assumption~\ref{assumption} can fail take $n=2$, $X_1$ uniform continuous on $[0,1]$, independent of $X_2$ that is uniform continuous on $[1/2,3/2]$, $\DD=\emptyset$ and $\CC=\{\{2\},\{1,2\}\}$, say. Then $F_{X_1\lor X_2}$ is constant on $(-\infty,1/2)$ but $F_{X_1}$ is not. 
\end{enumerate}
\end{remark}
For brevity of expression we shall assume the enforcement of Assumption~\ref{assumption} throughout, except where otherwise explicitly stated.

We conclude this section by fixing some general notation and agreeing on some conventions. For sets $B$ and $A$, $B^A$ will stand for the set of all functions from $A$ to $B$; $\vert A\vert$ for the size (cardinality) of $A$; and (by a slight abuse of notation) $2^{A}$ will denote the power set of a set $A$. For a distribution function $H$, $H(-\infty):=0$ and $H(\infty):=1$; whilst $H_l^{-1}$ will denote the left continuous \cite[p. 28, paragraph following proof of Proposition~4]{fristedt} inverse of $H$, $H_l^{-1}:(0,1)\to\mathbb{R}$, $H_l^{-1}(g):=\inf\{x\in \mathbb{R}:H(x)\geq g\}$ for $g\in (0,1)$. For a finite collection of extended real numbers $(a_i)_{i\in M}$, $\land_{i\in M}a_i:=\min\{a_i:i\in M\}$ and $\lor_{i\in M}a_i:=\max\{a_i:i\in M\}$. For $z\in \mathbb{R}^{[n]}$, $\sum z:=\sum_{i=1}^nz_i$. Next, we will adhere to the usual conventions that $\lor_\emptyset=-\infty$ \& $\land_\emptyset =+\infty$ or $\lor_\emptyset=0$ \& $\land_\emptyset =1$, depending on the context; and that $a^0\equiv 1$ in $a\in [-\infty,+\infty]$. Finally, as hitherto, for a random vector (in particular, variable) $Z$, $F_Z$ is its distribution function; and i.i.d. abbreviates ``independent and identically distributed''.

\section{The maxmin copula}\label{section:main}
\subsection{Step 1: Expressing $F_{X^{\lor\land}}$} We begin by finding the form of the distribution function for the system $X^{\lor\land}$. To this end let $x^{\lor\land}=((x^\lor_M)_{M\in \CC},(x^\land_M)_{M\in\DD})\in \RR^{ \CC}\times \RR^{\DD}$ be a collection of real numbers. By the very definition of the distribution function of a random vector
\begin{equation}\label{distribution}
    F_{X^{\lor\land}}(x^{\lor\land})=\PP\left(\bigcap_{M\in \CC}\left\{\lor_MX\leq x^\lor_M\right\}\cap \bigcap_{M\in \DD}\left\{\land_MX\leq x^\land_M\right\}\right)
\end{equation}
(where we understand $\cap_\emptyset=\Omega$). Denote now, provisionally, $\LL:=\cap_{M\in \CC}\{\lor_MX\leq x^\lor_M\}$. Using de Morgan's law and taking differences of events, \eqref{distribution} becomes equal to
$$\PP\left(\left(\bigcup_{M\in \DD}\left\{\land_MX> x^\land_M\right\}\right)^c\cap \LL\right)
=\PP(\LL)-\PP\left(\bigcup_{M\in\DD}\left\{\land_MX> x^\land_M\right\}\cap \LL\right).$$
An application of the inclusion-exclusion principle for the measure $\PP(\cdot\cap \LL)$ then turns this expression into
$$\PP(\LL)-\sum_{k=1}^{\vert\DD\vert}(-1)^{k+1}\sum_{\substack{I\subset \DD,\\ \vert I\vert=k}}\PP\left(\LL\cap \bigcap_{M\in I}\{\land_MX>x^\land_M\}\right).$$
Now $\LL=\cap_{i=1}^n\{X_i\leq \underset{\substack{M\in \CC,\\M\ni i}}{\bigwedge}x^\lor_M\}$, whilst $\bigcap_{M\in I}\{\land_MX>x^\land_M\}=\bigcap_{i=1}^n\left\{X_i>\underset{\substack{M\in I,\\ M\ni  i}}{\bigvee}x^\land_M\right\}$. We then see -- via another application of inclusion-exclusion -- that \eqref{distribution} in fact equals
$$
C\left(\left(F_i\left(\underset{\substack{M\in \CC,\\M\ni i}}{\bigwedge}x^\lor_M\right)\right)_{i=1}^n\right)-\sum_{k=1}^{\vert\DD\vert}(-1)^{k+1}\sum_{\substack{I\subset \DD,\\ \vert I\vert=k}}\sum_{z\in \{0,1\}^{[n]}}(-1)^{\sum z}
$$ $$
C\left(\left(F_i\left(\left(\underset{\substack{M\in \CC,\\M\ni i}}{\bigwedge}x^\lor_M\right)^{1-z_i}\left(\left(\underset{\substack{M\in \CC,\\M\ni i}}{\bigwedge} x^\lor_M\right) \land
\left(\underset{\substack{M\in I,\\ M\ni  i}}{\bigvee}x^\land_M\right)\right)^{z_i}\right)\right)_{i=1}^n\right).
$$
Assume now, for notational brevity, that $\DD\ne \emptyset$ (we defer the discussion of the (simpler) case when $\DD=\emptyset$ to Remark~\ref{remark:only_sup}). Since by the binomial theorem $1-\sum_{k=1}^{\vert \DD\vert}(-1)^{k+1}{\vert\DD\vert\choose k}=(1-1)^{\vert\DD\vert}= \delta_{0\vert\DD\vert}=0$, this equals further
$$
\sum_{k=1}^{\vert\DD\vert}\sum_{\substack{I\subset \DD,\\ \vert I\vert=k}}\sum_{\substack{z\in \{0,1\}^{[n]}\\ z\not\equiv 0}}\!\!\!\!(-1)^{k+\sum z}C\left(\left(F_i\left(\left(\underset{\substack{M\in \CC,\\M\ni i}}{\bigwedge}x^\lor_M\right)^{1-z_i}\left(\left(\underset{\substack{M\in \CC,\\M\ni i}}{\bigwedge}x^\lor_M\right)
\land
\left(\underset{\substack{M\in I,\\ M\ni  i}}{\bigvee}x^\land_M\right)\right)^{z_i}\right)\right)_{i=1}^n\right),
$$
which /using the fact that for any nondecreasing real-valued (in particular, any distribution) function $F$ and any reals $x$ and $y$, $F(x\land y)=F(x)\land F(y)$, and the same with $\lor$ replacing $\land$/ may finally be rewritten as
\begin{equation}\label{first_step}
\sum_{\substack{I\in 2^\DD\backslash\{\emptyset\}\\z\in \{0,1\}^{[n]}\backslash \{0\}}}\!\!\!\!\!\!\!\!\!\!\!(-1)^{\vert I\vert+\sum z}C\left(\left(\left(\underset{\substack{M\in \CC,\\M\ni i}} {\bigwedge}F_i(x^\lor_M)\right)^{1-z_i}\left(\left(\underset{\substack{M\in \CC,\\M\ni i}} {\bigwedge}F_i(x^\lor_M)\right)\land
    \left(\underset{\substack{M\in I,\\ M\ni  i}}{\bigvee}F_i(x^\land_M)\right)\right)^{z_i}\right)_{i=1}^n\right)
\end{equation}
(where now $\land_\emptyset$ should be understood as $1$, and $\lor_\emptyset$ as $0$).

\subsection{Step 2: Introduction of distortion functions and definition of maxmin copula}\label{subsection:distortion}
We have expressed in the previous subsection the joint distribution function $F_{X^{\lor\land}}(x^{\lor\land})$ given in \eqref{distribution} as a function of the marginal distributions \eqref{first_step}. In order to extract from this expression the relevant copula, we introduce the so-called distortion functions as in, say \cite{OmlRuz15} and \cite{DurOmlOraRuz15}.
We begin with the observation that, for each $M\in 2^{[n]}\backslash \{\emptyset\}$, \begin{equation}\label{dist:max}
F_{\lor_MX}=C\circ ((F_i\lor \mathbbm{1}_{[n]\backslash M}(i))_{i=1}^n),
\end{equation}
while thanks to inclusion-exclusion 
\begin{equation}\label{dist:min}
F_{\land_MX}=\sum_{j=1}^{\vert M\vert}(-1)^{j+1}\sum_{\substack{J\subset M,\\\vert J\vert=j}}C((F_i\lor \mathbbm{1}_{[n]\backslash J}(i))_{i=1}^n).
\end{equation}
It follows in particular that each $F_{\lor_MX}$ as well as $F_{\land_M X}$ is continuous (copulae being continuous \cite[p. 46, Theorem~2.10.7]{nelsen}).

Next, according to Lemma~\ref{lemma:key} from the appendix and thanks to the provisions of Assumption~\ref{assumption}, for $M$ from $\CC$, respectively $\DD$, $i\in M$, we may find a unique function $\Phi^{\lor}_{M,i}$, respectively $\Phi^{\land}_{M,i}$, that is continuous nondecreasing, mapping $[0,1]$ into $[0,1]$, and such that
$$F_i=\Phi^{\lor}_{M,i}\circ F_{\lor_M X},\text{ respectively }F_i=\Phi^{\land}_{M,i}\circ F_{\land_M X}.$$
It follows indeed from Remark~\ref{remark}\ref{remark:constructive} that \begin{equation}\label{distortions}
\Phi^\lor_{M,i}=F_i\circ (F_{\lor_M X})_l^{-1},\text{ respectively }\Phi^\lor_{M,i}=F_i\circ (F_{\lor_M X})_l^{-1},
\end{equation}
extended by zero at $0$ and unity at $1$ (recall $H_l^{-1}$ denotes the left-continuous inverse of a distribution function $H$). 

Finally, defining for $v^{\lor\land}=((v^\lor_M)_{M\in \CC},(v^\land_M)_{M\in\DD})\in [0,1]^{\CC}\times [0,1]^{\DD}$, the \textbf{maxmin copula} $C^{\lor\land}(v^{\lor\land})$ as equal to
\begin{equation}\label{main}
\sum_{\substack{I\in 2^\DD\backslash\{\emptyset\}\\z\in \{0,1\}^{[n]}\backslash \{0\}}}\!\!\!\!\!\!\!\!\!\!(-1)^{\vert I\vert+\sum z}C\left(\left(\left(\underset{\substack{M\in \CC,\\M\ni i}}{\bigwedge}\Phi_{M,i}^\lor(v^\lor_M)\right)^{1-z_i}\left(\left(\underset{\substack{M\in \CC,\\M\ni i}}{\bigwedge}\Phi_{M,i}^\lor(v^\lor_M)\right)\land
    \left(\underset{\substack{M\in I,\\ M\ni  i}} {\bigvee}\Phi_{M,i}^\land(v^\land_M)\right)\right)^{z_i}\right)_{i=1}^n\right)
\end{equation}
(again under the conventions $\land_\emptyset=1$, $\lor_\emptyset=0$), we see that
\begin{equation}\label{copula-relation}
F_{X^{\lor\land}}=C^{\lor\land}\circ ((F_{\lor_MX}\circ \pr^\lor_M)_{M\in \CC},(F_{\land_MX}\circ \pr^\land_M)_{M\in \DD}),
\end{equation}
where for $M$ from $\CC$, respectively $\DD$, $\pr^\lor_M: \RR^{\CC}\times \RR^{\DD}\to \RR$, respectively $\pr^\land_M:\RR^{\CC}\times \RR^{\DD}\to \RR$, is the canonical projection $((x^\lor_M)_{M\in \CC},(x^\land_M)_{M\in  \DD})\mapsto x^\lor_M$, respectively $((x^\lor_M)_{M\in \CC},(x^\land_M)_{M\in  \DD})\mapsto x^\land_M$.

\subsection{Main result}\label{subsection:main_result}
Since $C^{\lor\land}$ is manifestly continuous and since, thanks to the continuity of $F_{\lor_MX}$ and $F_{\land_MX}$, $\overline{\IM F_{\lor_MX}}=\overline{\IM F_{\land_MX}}=[0,1]$ for each $M\in 2^{[n]}\backslash \{\emptyset\}$, we conclude from \eqref{copula-relation} and the apposite properties of distribution functions (limits at $-\infty$, $+\infty$, inclusion-exclusion), that $C^{\lor\land}$ as given by \eqref{main} is in fact a copula -- it is then the unique (thanks to continuity of the marginals) copula for the system $X^{\lor\land}$ as given by Sklar's theorem.

\begin{theorem}\label{theorem}
Assuming $\DD\ne \emptyset$, $C^{\lor\land}$ as given by \eqref{main}-\eqref{distortions}-\eqref{dist:min}-\eqref{dist:max} is the copula for $X^{\lor\land}$.
\end{theorem}


\begin{remark}\label{remark:only_sup}
When $\DD=\emptyset$, but $\CC\ne\emptyset$, following the above reasoning, the copula for the system $X^\lor:=(\lor_M X)_{M\in\CC}$ is given by $$C^\lor(v^\lor)=C\left(\left(\bigwedge_{\substack{M\in \CC,\\ M\ni i}}\Phi^\lor_{M,i}(v^\lor_M)\right)_{i=1}^n\right),\quad v^\lor=(v^\lor_M)_{M\in \CC}\in [0,1]^{\CC}.$$
\end{remark}
Note that without the continuity of the $F_{\lor_MX}$ and $F_{\land_MX}$ and also of the $\Phi^{\lor}_{M,i}$ and $\Phi^{\land}_{M,i}$, it would have been very difficult (cumbersome at the very least) to check whether or not $C^{\lor\land}$ defines a copula. Indeed without access to the $\Phi$s it would not even have been clear how to define $C^{\lor\land}$ in the first place.  In view of Lemma~\ref{lemma:key} then, and at least within the confines of the methods used above, the provisions of Assumption~\ref{assumption} appear to be the minimal possible for the result to still obtain (but see the i.i.d. case).

\subsection{Special case -- $X$ is an independency} We assume here $C$ is the product copula, i.e. the $X_i$, $i\in [n]$, are independent. In this case Assumption~\ref{assumption} is verified if e.g. each $F_i$ is absolutely continuous with a strictly positive density, 
and we have for $u^{\lor\land}=((u^\lor_M)_{M\in \CC},(u^\land_M)_{M\in\DD})\in [0,1]^{\CC}\times [0,1]^{\DD}$ the following simplified form of the maxmin copula (which holds even if $\DD=\emptyset$, but $\CC\ne\emptyset$):

\begin{equation}\label{main:independent}
C^{\lor\land}(u^{\lor\land})=
\sum_{I\in 2^\DD}(-1)^{\vert I\vert}\prod_{i=1}^n\left(\left(\underset{\substack{M\in \CC,\\M\ni i}}{\bigwedge}\Phi^\lor_{M,i}(u^\lor_M)\right)-\left(\underset{\substack{M\in I,\\ M\ni  i}}{\bigvee}\Phi^\land_{M,i}(u^\land_M)\right)\right)^+.
\end{equation}

\subsection{i.i.d. case}
A remarkable (though not unexpected) further simplification occurs if we suppose in addition to independency that $F_i=F$ for all $i\in [n]$.  Note in this case Assumption~\ref{assumption} is verified for $F_i$, $i\in [n]$, that are continuous. Moreover, for each $M\in 2^{[n]}\backslash \{\emptyset\}$, $F_{\lor_MX}=\cdot^{\vert M\vert}\circ F$ and $1-F_{\land_MX}=(1-\cdot)^{\vert M\vert}\circ F$, so that we obtain for $M$ from $\CC$, respectively $\DD$, all $i\in M$, and on $(0,1)$, $$(F_{\lor_MX})^{-1}_l=F^{-1}_l\circ \sqrt[\vert M \vert]{\cdot}, \text{ respectively } (F_{\land_MX})^{-1}_l=F^{-1}_l\circ (1-\sqrt[\vert M\vert]{1-\cdot}).$$ Noting next that for a continuous distribution function $G$ one has $G\circ G^{-1}_l=\id_{(0,1)}$, the definition of the functions $\Phi$ implies
$$\Phi^{\lor}_{M,i}:=\sqrt[\vert M \vert]{\cdot},\text{ respectively } \Phi^{\land}_{M,i}:=(1-\sqrt[\vert M\vert]{1-\cdot}).$$
We conclude that in this case, for $u^{\lor\land}\in [0,1]^{\CC}\times [0,1]^{\DD}$,   $C^{\lor\land}(u^{\lor\land})$ is given by (again even if $\DD=\emptyset$, but $\CC\ne\emptyset$),
\begin{equation}\label{main:iid}
\mathbf{C^{\lor\land}}(u^{\lor\land}):=
\sum_{I\in 2^\DD}(-1)^{\vert I\vert}\prod_{i=1}^n\left(\underset{\substack{M\in \CC,\\M\ni i}}{\bigwedge}\sqrt[\vert M\vert]{u^\lor_M}+\underset{\substack{M\in I,\\ M\ni  i}}{\bigwedge}\sqrt[\vert M\vert]{1-u^\land_M}-1\right)^+,
\end{equation}
which in particular is universal (no longer dependent on $F$), and will be denoted (as indicated) simply by $\mathbf{C^{\lor\land}}$ below (though in general reference should be made in the notation to $\CC$, $\DD$ and $n$).

Indeed the latter observation, on the universality of $C^{\lor\land}$, allows us to make yet another. Maintain the i.i.d. case, but drop Assumption~\ref{assumption}. Possibly by extending the probability space, let $(Y^m)_{m\in\mathbb{N}}$ be a sequence of random vectors, independent of $X$, with independent continuous marginals, whose laws are  identical and converge weakly to $\delta_0$, as $m\to \infty$. For $m\in \mathbb{N}$, the random vector $X^m:=X+Y^m$ has continuous and i.i.d. marginals, and $X^m$ converges weakly to $X$ by the continuous mapping theorem, as $m\to\infty$. Therefore $F_{{(X^m)}^{\lor\land}}=\mathbf{C^{\lor\land}}\circ ((F_{\lor_MX^m}\circ \pr^\lor_M)_{M\in \CC},(F_{\land_MX^m}\circ \pr^\land_M)_{M\in \DD})$. Let $m\to \infty$. By the continuous mapping theorem and the continuity of $\mathbf{C^{\lor\land}}$, we conclude that  $F_{{X}^{\lor\land}}=\mathbf{C^{\lor\land}}\circ ((F_{\lor_MX}\circ \pr^\lor_M)_{M\in \CC},(F_{\land_MX}\circ \pr^\land_M)_{M\in \DD})$ at every point that is, component by component, a continuity point of each of the $F_{\lor_MX}$, $M\in \CC$, $F_{\land_MX}$, $M\in \DD$ (hence a continuity point of $F_{X^{\lor\land}}$ \cite[p. 164]{durrett}). So for each $x\in \mathbb{R}^{\CC}\times \mathbb{R}^{\DD}$, there exists a sequence of real numbers $(a_m)_{m\geq 1}$ nonincreasing to zero, such that for all $m\in \mathbb{N}$, $x+a_m\mathbbm{1}$ is, component by component, a continuity point of each of the $F_{\lor_MX}$, $M\in \CC$, $F_{\land_MX}$, $M\in \DD$. Thanks to the right-continuity of distribution functions, and the continuity of $\mathbf{C^{\lor\land}}$, we may pass to the limit and find that $F_{{X}^{\lor\land}}=\mathbf{C^{\lor\land}}\circ ((F_{\lor_MX}\circ \pr^\lor_M)_{M\in \CC},(F_{\land_MX}\circ \pr^\land_M)_{M\in \DD})$ holds at $x$, so everywhere, i.e. $\mathbf{ C^{\lor\land}}$, having already been established to be a copula (since in general there are continuous distribution functions $F$), is a copula for $X^{\lor\land}$ (though there may be others):

\begin{proposition}
Dropping Assumption~\ref{assumption}, but insisting $X$ is i.i.d., $\mathbf{C^{\lor\land}}$ as given by \eqref{main:iid}  is a copula for $X^{\lor\land}$.
\end{proposition}

\section{ Properties and applications }\label{section:properties}

\subsection{The POD property of the maxmin copula}\label{subsection:POD} It is intuitively appealing (at least when the starting vector $X$ is an independency) that taking maxima and minima should increase stochastic dependence. Let us make this precise in the context of the concept of positive orthant dependence \cite[p. 20, Subsection~2.1.1]{joe}. Remark that thanks to the continuity of the random variables in sight, positive orthant dependence of a random vector is
equivalent to the positive orthant dependence of its copula \cite[p. 187, Subsection~5.2.1]{nelsen}.

Now, in the context of Remark~\ref{remark:only_sup}, we observe that if $X$ is an iid vector, then $X^\lor$ is positive orthant dependent (POD). Indeed we have the slightly more general result:
\begin{proposition}\label{proposition:POD-for-maxima}
If $C$ is POD, and if for all $M\in\CC$, $x\in \mathbb{R}$ the probability of $\{X_i\leq x\text{ for all }i\in M\}$ is $\leq$ (hence $=$) $\prod_{i\in M}\PP(X_i\leq x)$, then $C^\lor$ is POD.
\end{proposition}
\begin{proof}
The POD property of $C^\lor$, by definition, means $$C\left(\left(\bigwedge_{\substack{M\in \CC,\\ M\ni i}}\Phi^\lor_{M,i}(v^\lor_M)\right)_{i=1}^n\right)\geq \prod_{M\in\CC}v^\lor_M,\text{ for all }v^\lor=(v^\lor_M)_{M\in \CC}\in [0,1]^{\CC}.$$ Since $C$ is POD, it is sufficient to observe $$\prod_{i=1}^n\bigwedge_{\substack{M\in \CC,\\M\ni i}}\Phi^\lor_{M,i}(v^\lor_M)\geq \prod_{M\in\CC}v^\lor_M,\text{ for all }v^\lor=(v^\lor_M)_{M\in \CC}\in [0,1]^{\CC}.$$ But this follows from (indeed is equivalent to) $$\prod_{i\in M}\Phi^\lor_{M,i}(v)\geq v,\text{ for all }v\in (0,1)\text{ or for all }v\in [0,1],\text{ }M\in \CC.$$ The latter in turn is of course equivalent to the stated condition (writing $v=F_{\lor_MX}(x)$).
\end{proof}
There is, in view of the preceding result, no need to comment on the absurdity which would result from insisting, in general, on the negative orthant dependence of the system of maxima.

For the maxmin system, we have the following generalization, in the independent case:

\begin{proposition}
If $X$ is an independency (but dropping Assumption~\ref{assumption}), then $X^{\lor\land}$ is POD.
\end{proposition}
\begin{proof}
It is assumed without loss of generality $\CC=\DD=2^{[n]}\backslash \{\emptyset\}$. We prove by induction on $n\in \mathbb{N}$. Assume first $n=1$. Then for $\{x^\lor,x^\land\}\subset \mathbb{R}$, of course, $\PP(X_1\leq x^\land,X_1\leq x^\lor)=\PP(X_1\leq x^\land\land x^\lor)\geq \PP(X_1\leq x^\land)\PP(X_1\leq x^\lor)$. Assume now $C^{\lor\land}$ is POD for $n\in \mathbb{N}$. Let $x^{\lor\land}=((x^{\lor}_M)_{M\in 2^{[n+1]}\backslash \{\emptyset\}},(x^\land_M)_{M\in  2^{[n+1]}\backslash \{\emptyset\}})\in \mathbb{R}^{2^{[n+1]}\backslash \{\emptyset\}}\times \mathbb{R}^{2^{[n+1]}\backslash \{\emptyset\}}$. Thanks to Tonelli's theorem (using independence), we find $$\PP(X^{\lor\land}\leq x^{\lor\land})=\int \PP_{X_{n+1}}(dx)$$ $$\PP\left(\cap_{M\in 2^{[n]}\backslash \{\emptyset\}}\{\lor_M X\leq x^\lor_M\land x^\lor_{M\cup \{n+1\}} \}\cap\left\{\land_M X\leq
\begin{cases}
x_M^\land, &\text{if }x\leq x_{M\cup \{n+1\}}^\land\\
x_M^\land\land x_{M\cup\{n+1\}}^\land, & \text{otherwise}
\end{cases}
\right\} \right)$$
$$\mathbbm{1}_{\left(-\infty,x_{\{n+1\}}^\land\land \left(\bigwedge_{M\in 2^{[n]}} x_{M\cup \{n+1 \}}^\lor\right)\right]}(x),
$$
which according to the induction hypothesis is certainly $\geq$ to $$\int \PP_{X_{n+1}}(dx)\prod_{M\in 2^{[n]}\backslash \{\emptyset\}} \PP\left(\land_MX\leq \begin{cases}
x_M^\land, &\text{if }x\leq x_{M\cup \{n+1\}}^\land\\
x_M^\land\land x_{M\cup\{n+1\}}^\land, & \text{otherwise}
\end{cases}\right)\PP(\lor_MX\leq x^\lor_M\land x^\lor_{M\cup \{n+1\}})$$ $$\mathbbm{1}_{\left(-\infty,x_{\{n+1\}}^\land\land \left(\bigwedge_{M\in 2^{[n]}} x_{M\cup \{n+1 \}}^\lor\right)\right]}(x)=$$

$$=\int \PP_{X_{n+1}}(dx)\prod_{M\in 2^{[n]}\backslash \{\emptyset\}} \PP(\land_MX\leq x^\land _M,x\land (\land_{M}X)\leq x^\land_{M\cup \{n+1\}})\PP(\lor_MX\leq x^\lor_M\land x^\lor_{M\cup \{n+1\}})$$ $$\mathbbm{1}_{\left(-\infty,x_{\{n+1\}}^\land\land \left(\bigwedge_{M\in 2^{[n]}} x_{M\cup \{n+1 \}}^\lor\right)\right]}(x) \geq $$

(noting that for a random variable $A$ and $\{a,x,c\}\subset \mathbb{R}$, $\PP(A\leq a,x\land A\leq c)\geq \PP(A\leq a)\PP(A\land x\leq c)$, whilst $\PP(A\leq a\land c)\mathbbm{1}_{(-\infty,c]}(x)\geq \PP(A\leq a)\PP(A\lor x\leq c)$) $$\geq \int \PP_{X_{n+1}}(dx)\PP(x\leq x_{n+1}^\lor)\PP(x\leq x^\land_{n+1})$$ $$\prod_{M\in 2^{[n]}\backslash \{\emptyset\}} \PP(\land_MX\leq x^\land _M)\PP(x\land (\land_{M}X)\leq x^\land_{M\cup \{n+1\}})\PP(\lor_MX\leq x^\lor_M)\PP(x\lor (\lor_MX)\leq x^\lor_{M\cup \{n+1\}})$$ and we conclude via an application of Lemma~\ref{lemma:reverse-Cauchy-Schwartz} (nonincreasingness, nonnegativity and boundedness of the relevant functions is clear, whilst their left-continuity follows e.g. by bounded convergence),  and then again Tonelli's theorem (exploiting independence).
\end{proof}

\subsection{Supports and zero sets}\label{subsection:zero}
Due to a kind of degeneracy of the maxmin system $X^{\lor\land}$ (its components, of which there can be as many as $2\cdot (2^n-1)$, attain at most $n$ distinct values), a characterization of the support of the maxmin copula $C^{\lor\land}$ does not appear the most relevant (nor, indeed, feasible in general, if one insists further on it being tractable).

We remark only as follows:
\begin{proposition}
If $X$ is an independency, then the zero set of $C^{\lor\land}$ equals $[0,1]^{\CC\times \DD}\backslash (0,1]^{\CC\times\DD}$.
\end{proposition}
\begin{proof}
For $M\in \CC$ (respectively $M\in \DD$), let $x^\lor_M>\inf\supp(\PP_{\lor_MX})=\lor_{i\in M}\alpha_i$ (respectively $x^\land_M>\inf\supp(\PP_{\land_MX})=\land_{i\in M} \alpha_i$), where for $i\in [n]$, $\alpha_i:=\inf\supp(\PP_{X_i})$; then $\PP(X^{\lor\land}\leq ((x^\lor_M)_{M\in \CC},(x^\land_M)_{M\in \DD}))>0$. Conclude via the nondecreasingness of $C^{\lor\land}$ in each coordinate, and the continuity of the relevant distribution functions.
\end{proof}

\subsection{Connection to order statistics} \label{subsection:order}
Note that for $i\in [n]$, the $i$-th order statistic of a random vector $(Y_1,\ldots,Y_n)$ can be expressed as $Y_{(i)}=\land\{\lor_M Y:M\in 2^{[n]}, \vert M\vert=i\}$. Moreover, when the $Y_i$, $i\in [n]$, are i.i.d. with continuous distribution functions $F$, then the random vector $X=(\lor_M Y:M\in 2^{[n]}\backslash \{\emptyset\})$ consists of continuous random variables, it has the copula $$C(u)=\prod_{i=1}^n\left(\underset{\substack{M\in 2^{[n]}\backslash \{\emptyset\},\\M\ni i}}{\bigwedge}\sqrt[\vert M\vert]{u_M}\right),\quad u=(u_M)_{M\in 2^{[n]}\backslash \{\emptyset\}}\in [0,1]^{2^{[n]}\backslash \{\emptyset\}}$$ and with $\DD=\{\{M\in 2^{[n]}\backslash \{\emptyset\}:\vert M\vert =k\}:k\in [n]\}$, $\CC=\emptyset$, satisfies the provisions of Assumption~\ref{assumption}. We also have the distribution function of the $i$-th order statistic given as $F_{(i)}=\sum_{k=i}^n{n\choose k}F^k(1-F)^{n-k}=O_i\circ F$, where \begin{equation}\label{os:aux}
O_i=\left(u\mapsto \sum_{k=i}^n{n\choose k}u^k(1-u)^{n-k},u\in [0,1]\right)
\end{equation}
is an increasing bijection\footnote{It follows e.g. from the fact that in the case of independent, identically --- uniformly on $[0,1]$ --- distributed random variables, the distribution functions of the order statistics are strictly increasing on $[0,1]$.} of $[0,1]$; whilst for a $k$-element subset $M$ of $[n]$, of course, $F_{\lor_M Y}=F^k$. Hence the relevant distortion functions are given for $k\in [n]$ and $k$-element subsets $M$ of $[n]$, as $\zeta_k:=\cdot^{k}\circ O_k^{-1}$. It follows that \eqref{main} gives an explicit expression for the copula of the order statistics $(Y_{(1)},\ldots,Y_{(n)})$, specifically, denoting it by $C^{\text{OS}}$, it is given by (for $v\in \mathbb{R}^n$): $$C^{\text{OS}}(v)=\sum_{k=1}^n\sum_{\substack{I\subset [n]\\ \vert I\vert=k}}\sum_{\substack{z\in \{0,1\}^{2^{[n]}\backslash \{\emptyset\}}\\ z\not\equiv 0}}(-1)^{k+\sum z}\prod_{i=1}^n \bigwedge_{\substack{M\in 2^{[n]}\backslash \{\emptyset\},\\ M\ni i}}\sqrt[\vert M\vert]{\zeta_{\vert M\vert}(v_{\vert M\vert})\mathbbm{1}_{I}(\vert M\vert)}^{z_M}.$$ This can be substantially simplified. Indeed, since the term in the sum of the preceding display is non-zero only if for all $M\in 2^{[n]}\backslash \{\emptyset\}$, $\vert M\vert\notin I$ implies $z_M=0$, we obtain, denoting

\begin{equation}\label{OS:[[]]}
[[z]]:=\vert \{\vert M\vert:M\in z^{-1}(\{1\})\}\vert\text{ for a function }z,
\end{equation}
that
$$C^{\text{OS}}(v)=\sum_{k=1}^n\sum_{\substack{I\subset [n]\\ \vert I\vert=k}}\sum_{\substack{z\in \{0,1\}^{2^{[n]}\backslash \{\emptyset\}}\\ z\not\equiv 0\\ \vert M\vert\notin I\Rightarrow z_M=0}}(-1)^{k+\sum z}\prod_{i=1}^n \bigwedge_{\substack{M\in 2^{[n]}\backslash \{\emptyset\},\\ M\ni i\\ z_M=1}}O_{\vert M\vert}^{-1}(v_{\vert M\vert})$$
$$=\sum_{\substack{z\in \{0,1\}^{2^{[n]}\backslash \{\emptyset\}}\\ z\not\equiv 0}}(-1)^{\sum z}\sum_{k=1}^n(-1)^k\sum_{\substack{I\subset [n]\\\vert I\vert=k\\z_M=1\Rightarrow \vert M\vert\in I}}\prod_{i=1}^n \bigwedge_{\substack{M\in 2^{[n]}\backslash \{\emptyset\},\\ M\ni i\\ z_M=1}}O_{\vert M\vert}^{-1}(v_{\vert M\vert})$$
$$=\sum_{\substack{z\in \{0,1\}^{2^{[n]}\backslash \{\emptyset\}}\\ z\not\equiv 0}}(-1)^{\sum z}\left(\prod_{i=1}^n \bigwedge_{\substack{M\in 2^{[n]}\backslash \{\emptyset\},\\ M\ni i\\ z_M=1}}O_{\vert M\vert}^{-1}(v_{\vert M\vert})\right)\sum_{k=1}^n(-1)^k\sum_{\substack{I\subset [n]\\\vert I\vert=k\\z_M=1\Rightarrow \vert M\vert\in I}}1$$
$$=\sum_{\substack{z\in \{0,1\}^{2^{[n]}\backslash \{\emptyset\}}\\ z\not\equiv 0}}(-1)^{\sum z}\left(\prod_{i=1}^n \bigwedge_{\substack{M\in 2^{[n]}\backslash \{\emptyset\},\\ M\ni i\\ z_M=1}}O_{\vert M\vert}^{-1}(v_{\vert M\vert})\right)\sum_{k=[[ z]]}^n(-1)^k{n-[[z]]\choose k-[[z]]}$$
$$=\sum_{\substack{z\in \{0,1\}^{2^{[n]}\backslash \{\emptyset\}}\\ z\not\equiv 0}}(-1)^{\sum z}\left(\prod_{i=1}^n \bigwedge_{\substack{M\in 2^{[n]}\backslash \{\emptyset\},\\ M\ni i\\ z_M=1}}O_{\vert M\vert}^{-1}(v_{\vert M\vert})\right)(-1)^{[[z]]}\delta_{0(n-[[z]])};$$
finally $C^{\text{OS}}(v)$ is seen to be equal to

\begin{equation}\label{equation:order_statistic_copula}
\mathbf{C^{\text{OS}}}(v):=(-1)^n\sum_{\substack{z\in \{0,1\}^{2^{[n]}\backslash \{\emptyset\}},\\ [[z]]=n}}\!\!\!\!\!\!\!(-1)^{\sum z}\prod_{i=1}^n \bigwedge_{\substack{M\in 2^{[n]}\backslash \{\emptyset\},\\ M\ni i\\ z_M=1}}O_{\vert M\vert}^{-1}(v_{\vert M\vert}).
\end{equation}
The (so introduced) \textbf{order statistics copula} $\mathbf{C^{\text{OS}}}$ is of course universal (independent of $F$), as previously observed \cite{averous,navarro}. Consequently, just as with the maxmin copula in the i.i.d. scenario, and with an entirely analogous proof, we may drop the requirement that $F$ be continuous, and still $\mathbf{C^{\text{OS}}}$ remains a copula of the vector of the order statistics.

\begin{proposition}\label{proposition:order_statistic}
Let $Y$ be an i.i.d. random vector with values in $\mathbb{R}^n$. Then $\mathbf{C^{\text{OS}}}$, as given by \eqref{equation:order_statistic_copula}-\eqref{OS:[[]]}-\eqref{os:aux}, is a copula for its order statistics vector $(Y_{(1)},\ldots,Y_{(n)})$.
\end{proposition}

\subsection{Successive formations of maxmin systems and connection to survival analysis}\label{subsection:survival} Let $X$, $\CC$, $\DD$, $X^{\lor\land}$ and $C^{\lor\land}$ be as in Sections~\ref{section:preliminaries} and~\ref{section:main}. We call the mapping $C\mapsto C^{\lor\land}$ a \textbf{maxmin transformation}, where $X$, $\CC$, $\DD$, and $X^{\lor\land}$ are understood and omitted by an abuse of notation. Now if the resulting vector $X^{\lor\land}$ again satisfies Assumption~\ref{assumption} (with the obvious change of $n$ and newly defined sets $\CC$ and $\DD$), then we can repeat the procedure and call the composition of the two transformations a \textbf{second order maxmin transformation}. It is possible that this procedure can be continued even further and brings us eventually to the composition of $m$ consecutive maxmin transformations to be called an \textbf{$m$-fold maxmin transformation} of the starting system and its final copula denoted by $C^{\lor\land}_{(m)}$ and called an \textbf{$m$-fold maxmin copula}.

Now, in principle, Theorem~\ref{theorem} assures us a maxmin copula can be obtained in closed form under Assumption~\ref{assumption}. In the context of applying the maxmin transformation consecutively, and sacrificing some generality to the benefit of simplicity, the distribution of $X$ will thus be said to \textbf{verify Assumption~\ref{assumption} up to order $m$}, if $X$ and its derivatives $X^{\lor\land}_{(i)}$ for all $i \in [m-1]$, and for any choice of $\CC$ and $\DD$ in each step, satisfy the provisions of Assumption~\ref{assumption}. Then Remark~\ref{remark:stable}\ref{remark:stable:b} gives us a reasonably general class of distributions that verifies Assumption~\ref{assumption} up to indeed any order, and for a system of the kind we obviously have  by a successive application of Theorem~\ref{theorem}:

\begin{quote}
\emph{If the distribution of $X$ verifies Assumption~\ref{assumption} up to order $m$, then the copula $C^{\lor\land}_{(m)}$ is obtainable, at least in principle, in closed form.}
\end{quote}
Some remarks. (1) We say ``at least in principle'', since with increasing $m$ the computational complexity of the evaluation of $C^{\lor\land}_{(m)}$ becomes a legitimate concern. (2) That being said, simplifications can occur in the process, e.g. the order statistics copula $\mathbf{C^{\mathrm{OS}}}$ was actually obtained precisely through a second order maxmin transformation.  (3) A question open to future research is this: if $X$ verifies Assumption~\ref{assumption}, does it automatically verify it up to any order?


Consider now the shock model depicted in detail in the Introduction. We have a (component) random vector $Y= (Y_i)_{i=1}^n$, and a (shock) random vector $Z=(Z_j)_{j=1}^m$, with joint copula $C$ for the vector $X:=(Y,Z)$. We introduce $\mathrm{opt}_{i,j}$ by
\[
    \mathrm{opt}_{i,j}:=\left\{
                                     \begin{array}{ll}
                                       \min, & \hbox{if the shock $j$ is detrimental to component $i$;} \\
                                       \max, & \hbox{if the shock $j$ is beneficial to component $i$.}
                                     \end{array}
                                   \right.
\]
for $i\in [n]$, $j\in [m]$.  In general the order of the shocks to each of the components matters (since $\lor$ does not commute with $\land$). For this reason, for each $i\in [n]$, we fix a permutation $w_i\in [m]^{[m]}$ of $[m]$ which determines the order in which the shocks affect the $i$-th component. Then the resulting lifetime of the $i$-th component, after $k$ shocks have affected it, is given by
\[
    Y_i^k:= \mathrm{opt}_{i,w_i(k)}\{\mathrm{opt}_{i,w_i(k-1)}\{\cdots\mathrm{opt}_{i,w_i(1)}\{Y_i,Z_{w_i(1)}\}\cdots,
    Z_{w_i(k-1)}\},Z_{w_i(k)}\}.
\]
Furthermore, define for each $k\in [m]$ the system
\[
    X^{\lor\land}_{(k)}:=((Y_i ^k)_{i\in[n]},Z)
\]
except for $k=m$, when $X^{\lor\land}_{(m)}:=(Y_i^m)_{i\in [n]}$. A little thought reveals that each of the copulas $C^{\lor\land}_{(k)}$ of $X^{\lor\land}_{(k)}$ is obtained via a maxmin transformation from the one with index $k-1$ (if we denote additionally $X^{\lor\land}_{(0)}:=(Y,Z)$ and $C^{\lor\land}_{(0)}:=C$), so that $C^{\lor\land}_{(m)}$ is an $m$-fold maxmin copula by an inductive argument on $k$, provided the provisions of Assumption~\ref{assumption} are met at each step. The following simple observation follows from the above:

\begin{quote}
\emph{If the distribution of the random vector $(Y,Z)$ verifies Assumption~\ref{assumption} up to order $m$, then the copula of the resulting lifetimes $(Y_i^m)_{i\in [n]}$ is an $m$-fold maxmin copula.}
\end{quote}
Obvious generalizations of the above are possible, e.g. not all the shocks need affect every component, the same shock may affect the same component several times etc. The same reasoning applies, only the notation changes. 

\subsection{Low-dimensional examples}\label{subsection:low-dimensional}
Let $X_1$, $X_2$, $X_3$ be i.i.d.

\begin{example}
Let $T_A:=X_1\land X_2$, $T_B:=X_1\lor X_2$. Then the copula $C:[0,1]^2\to [0,1]$ for the random vector $(T_A,T_B)$ is given by (using \eqref{main:iid}): $$C(a,b)=b-\left(\left(\sqrt{b}+\sqrt{1-a}-1\right)^+\right)^2.$$ As a check we may obtain this copula also using \eqref{equation:order_statistic_copula}; it yields the expression $2\sqrt{b}(\sqrt{b}\land (1-\sqrt{1-a}))-(\sqrt{b}\land (1-\sqrt{1-a}))^2$, which is of course the same.
\end{example}

\begin{example}
Let $T_A:=X_1\land X_2$, $T_B:=X_1\lor X_2\lor X_3$. Then the copula $C:[0,1]^2\to [0,1]$ for the random vector $(T_A,T_B)$ is given by (using \eqref{main:iid}): $$C(a,b)=b-\sqrt[3]{b}\left(\left(\sqrt[3]{b}+\sqrt{1-a}-1\right)^+\right)^2.$$
\end{example}

\begin{example}
Let $T_A:=X_1\land X_2\land X_3$, $T_B:=X_1\lor X_2$. Then the copula $C:[0,1]^2\to [0,1]$ for the random vector $(T_A,T_B)$ is given by (using \eqref{main:iid}): $$C(a,b)=b-\sqrt[3]{1-a}\left(\left(\sqrt{b}+\sqrt[3]{1-a}-1\right)^+\right)^2.$$
\end{example}

\bibliographystyle{plain}
\bibliography{Biblio_copulae}

\begin{thebibliography}{10}

\bibitem{averous}
J.~Av\'erous, C.~Genest, and S.~C. Kochar.
\newblock On the dependence structure of order statistics.
\newblock {\em Journal of Multivariate Analysis}, 94(1):159 -- 171, 2005.

\bibitem{mendes}
B.~V. de~Melo~Mendes and M.~A. Sanfins.
\newblock The limiting copula of the two largest order statistics of
  independent and identically distributed samples.
\newblock {\em Brazilian Journal of Probability and Statistics}, 21:85--101,
  2007.

\bibitem{anjos}
U.~U. dos Anjos, N.~Kolev, and N.~I. Tanaka.
\newblock Copula associated to order statistics.

\bibitem{DurOmlOraRuz15}
F.~Durante, M.~Omladi\v{c}, L.~Ora\v{z}em, and N.~Ru\v{z}i\'c.
\newblock Shock models with recovery option via the maxmin copulas.

\bibitem{durrett}
R.~Durrett.
\newblock {\em Probability: Theory and Examples}.
\newblock Thomson, Brooks Cole, third edition, 2005.

\bibitem{fristedt}
B.~E. Fristedt and L.~F. Gray.
\newblock {\em A Modern Approach to Probability Theory}.
\newblock Probability and Its Applications. Birkh{\"a}user Boston, 2013.

\bibitem{jaworski}
P.~Jaworski and T.~Rychlik.
\newblock On distributions of order statistics for absolutely continuous
  copulas with applications to reliability.
\newblock {\em Kybernetika}, 6(6), 2008.

\bibitem{joe}
H.~Joe.
\newblock {\em Multivariate Models and Multivariate Dependence Concepts}.
\newblock Chapman \& Hall/CRC Monographs on Statistics \& Applied Probability.
  Taylor \& Francis, 1997.

\bibitem{MO67}
A.~W. Marshall and I.~Olkin.
\newblock A multivariate exponential distribution.
\newblock {\em Journal of the American Statistical Association}, 62:30--44,
  1967.

\bibitem{navarro}
J.~Navarro and F.~Spizzichino.
\newblock On the relationships between copulas of order statistics and marginal
  distributions.
\newblock {\em Statistics \& Probability Letters}, 80(5–6):473 -- 479, 2010.

\bibitem{nelsen}
R.~B. Nelsen.
\newblock {\em An Introduction to Copulas}.
\newblock Lecture notes in statistics. Springer, 1999.

\bibitem{OmlRuz15}
M.~Omladi\v{c} and N.~Ru\v{z}i\'c.
\newblock Shock models with recovery option via the maxmin copulas.
\newblock {\em Fuzzy Sets and Systems}, (in press), 2015.

\bibitem{schmitz}
V.~Schmitz.
\newblock {Revealing the dependence structure between $X_{(1)}$ and $X_{(n)}$}.
\newblock {\em Journal of Statistical Planning and Inference}, 123(1):41 -- 47,
  2004.

\end{thebibliography}

\appendix

\section{Auxiliary results}\label{section:auxiliary}

\begin{lemma}[Transformations of distribution functions]\label{lemma:key}
Let $F$ and $G$ be two distribution functions (right-continuous nondecreasing $[0,1]$-valued functions on $\mathbb{R}$ with respective limits $0$ and $1$ at $-\infty$ and $+\infty$). Then there exists a continuous nondecreasing (respectively strictly increasing) $\Phi:[0,1]\to [0,1]$ such that $F=\Phi\circ G$, if and only if:
$$\text{every interval of constancy of }G\text{ is an interval of constancy of }F$$
and
$$\text{ every point of discontinuity of  }F\text{ is}\text{ a point of discontinuity of }G$$  (respectively $G$ and $F$ share their intervals of constancy and points of discontinuity).

Moreover, such $\Phi$ is unique whenever (respectively if and only if) $\overline{\IM G}=[0,1]$, i.e. $G$ is continuous.
\end{lemma}
\begin{remark}\label{remark}
\leavevmode
\begin{enumerate}[(i)]
\item\label{remark:constructive} The proof is constructive in the sufficiency part. Indeed, if $G$ is continuous, then $\Phi=F\circ G_l^{-1}$, extended by zero at $0$ and unity at $1$. More generally, $\Phi$ may be chosen so that it agrees with $F\circ G_l^{-1}$ on $(0,1)\cap\IM G$.
\item Such $\Phi$ always automatically satisfies $\Phi(0)=0$ and $\Phi(1)=1$ /test $F=\Phi\circ G$ against the limits at $\pm\infty$, and use the nondecreasingness of $\Phi$/.
\item Dropping the continuity requirement, we have the following simpler characterization, which complements the result of Lemma~\ref{lemma:key}:
\begin{quote}
Let $F$ and $G$ be two distribution functions. There exists a nondecreasing $\Phi:[0,1]\to [0,1]$ such that $F=\Phi\circ G$, if and only if:
$$\text{every interval of constancy of }G\text{ is an interval of constancy of }F.$$
Such $\Phi$ is unique whenever $(0,1)\subset \IM G$.
\end{quote}
\begin{proof}
Necessity and uniqueness is clear. Sufficiency follows by taking $\Phi=F\circ G_l^{-1}$, extending it by zero at $0$ and unity at $1$ (cf. proof below).

Interestingly enough, taking $G=\id_{[0,1/2)}\mathbbm{1}_{[0,1/2)}+(1/2+1/2\id_{[1/2,1]})\mathbbm{1}_{[1/2,1]}+\mathbbm{1}_{(1,\infty)}$ and $F=\id_{[0,1]}\mathbbm{1}_{[0,1]}+\mathbbm{1}_{(1,\infty)}$ these two distribution functions share their intervals of constancy, $(0,1)\backslash \IM G$ is a non-degenerate interval, and yet the unique nondecreasing $\Phi:[0,1]\to [0,1]$ for which $F=\Phi\circ G$ is not strictly increasing (a quick analysis reveals indeed that $\Phi$ maps $x\mapsto x$ for $x\in [0,1/2)$, $x\mapsto 2x-1$ for $x\in [3/4,1]$ and hence $x\mapsto 1/2$ for $x\in [1/2,3/4)$).
\end{proof}
\end{enumerate}
\end{remark}
\begin{proof}
The remark regarding uniqueness is clear (respectively, in the necessity part, will be clear, from the construction following). Necessity of the conditions is clear. To prove sufficiency we construct $\Phi$ as follows.

Let $G_l^{-1}$ be the left-continuous inverse of $G$ (recall $G_l^{-1}:(0,1)\to\mathbb{R}$, $G_l^{-1}(g)=\inf\{x\in \mathbb{R}:G(x)\geq g\}$ for $g\in (0,1)$) that is clearly nondecreasing. Then define for $x\in [0,1]$, $\Phi(x):=0$ if $x=0$; $\Phi(1):=1$ if $x=1$; and otherwise $\Phi(x):=F(G_l^{-1}(x))$ if $x$ belong to $\IM G\cap (0,1)$; finally $\Phi(x):=(F\circ G_l^{-1})(\ux -)+\frac{x-\ux}{\ox-\ux}[(F\circ G_l^{-1})(\ox)-(F\circ G_l^{-1})(\ux-)]$ if $x\in (0,1)\backslash \IM G$ (where we understand $(F\circ G_l^{-1})(1):=1$ and $(F\circ G_l^{-1})(0-):=0$; and for $x\in (0,1)\backslash \IM G$, $\ox:=G(G^{-1}_l(x))$ and $\ux:=G(G^{-1}_l(x)-)$ /so that $[x,x+\epsilon)\cap \IM G=\emptyset$ for some $\epsilon>0$, thanks to the right-continuity of $G$, and necessarily $(\ux,\ox)$ is a maximal non-degenerate open interval of $(0,1)\backslash \IM G$, with $\ox\in \IM G$, $\ux\in \overline{\IM G}$/).\footnote{Any nondecreasing continuous interpolation between $(\ux,(F\circ G_l^{-1})(\ux-))$ and $(\ox,(F\circ G_l^{-1})(\ox))$ would of course also do.}

We verify at once that $F=\Phi\circ G$; let $x\in \mathbb{R}$. If $G(x)=1$, then $F(x)=1$ also, since $\lim_{\infty}F=1$ and every interval of constancy of $G$ is an interval of constancy of $F$ -- the equality $F(x)=\Phi(G(x))$ indeed holds. Similarly if $G(x)=0$, then $F(x)=0$, and again the desired equality obtains. Otherwise let $z:=\inf\{y\in \mathbb{R}:G(y)=G(x)\}\in \mathbb{R}$. Then, since every interval of constancy of $G$ is an interval of constancy of $F$, in order to check $F(x)=\Phi(G(x))$ it suffices to find $F(z)=\Phi(G(z))$. But clearly $G^{-1}_l(G(z))=z$. So it will be enough to argue then that $\Phi(G(z))=F(G_l^{-1}(G(z)))$, but this follows from the obvious observation that $G(z)\in (0,1)\cap \IM G$.

Before proceeding, two remarks.

(a) For $0<a<b<1$, $(a,b)$ is a maximal non-degenerate open interval of $(0,1)\backslash \IM G$, if and only if it is a  maximal non-degenerate open interval of constancy of $G_l^{-1}$.

(b) If $\IM G\ni u<x\notin \IM G$ are from $(0,1)$, then $F(G_l^{-1}(u))\leq (F\circ G_l^{-1})(\ux-)$. For, if $\underline{x}\notin \IM G$, then $u<\ux$ and so by Remark~(a) $G_l^{-1}(u)<G_l^{-1}(\ux)$. If, however, $\ux\in \IM G$, then $G$ is locally constant strictly to the left of $G^{-1}_l(x)$, hence so is $F$, and then if $G$ has a jump at $G_l^{-1}(\ux)$, $(F\circ G_l^{-1})(\ux-)=F(G_l^{-1}(\ux))$, whilst if it does not, then neither does $F$, hence again $(F\circ G_l^{-1})(\ux-)=F(G_l^{-1}(\ux))$.

Now it is clear from the construction and Remark (b) that $\Phi$ is nondecreasing. [In the respective instance we have indeed that the intervals of constancy of $G_l^{-1}$ coincide with those of $F\circ G_l^{-1}$ /for, if $0<a<b<1$ and $G_l^{-1}(a)<G_l^{-1}(b)$, then $G(G_l^{-1}(a))<G(G_l^{-1}(b))$, hence $F(G_l^{-1}(a))<F(G_l^{-1}(b))$/. Then take a maximal open non-degenerate interval of constancy $(a,b)$ for $F\circ G_l^{-1}$, equivalently $G_l^{-1}$:
\begin{itemize}
\item if $a=0$, $ G\circ G_l^{-1}(b)\geq b>0$, so $F\circ G_l^{-1}(b)>0$, hence $F\circ G_l^{-1}(a-)=0< F\circ G_l^{-1}(b)$;
\item otherwise if $a\notin \IM G$, $G$ must have a jump over $(a,b)$ at $G_l^{-1}(a)=G_l^{-1}(b)$, hence $F$ must have a jump at $G_l^{-1}(a)$ also, whilst $a$ is a left increase point of $G_l^{-1}$, and thus $F\circ G_l^{-1}(a-)<F\circ G_l^{-1}(a)\leq F\circ G_l^{-1}(b)$ /also works for $b=1$: in that case, $1\in \IM G$ and the definition of $G_l^{-1}(g)$ extends naturally to include $g=1$, with $F(G_l^{-1}(1))=1$/;
\item finally if $a\in \IM G$, $G$ jumps over $(a,b)$ at $G_l^{-1}(b)$, is locally constant at level $a$ strictly to the left of $G_l^{-1}(b)$, $G_l^{-1}(a)<G_l^{-1}(b)$, $G(G_l^{-1}(a))=a<b=G(G_l^{-1}(b))$ and hence $F\circ G_l^{-1}(a-)\leq F\circ G_l^{-1}(a)<F\circ G_l^{-1}(b)$  /again also works for $b=1$/.
\end{itemize}
It follows that now $\Phi$ is strictly increasing.]

Two more remarks. Let $x\in (0,1)$.

(1) If $F\circ G_l^{-1}$ is not right-continuous  at $x$, then $x\in \IM G$ and $G$ has a jump at $G_l^{-1}(x+)$. For, if $F\circ G_l^{-1}$ is not right-continuous at $x$, then we must have  $G_l^{-1}(x+)>G_l^{-1}(x)$, by the right-continuity of $F$, and necessarily $G(G_l^{-1}(x))=x$ (otherwise $G_l^{-1}$ would be locally constant to the right of $x$). It follows that $G$, hence $F$, has to be constant on $[G_l^{-1}(x),G_l^{-1}(x+))$. Now if $G$ would not have a jump at $G_l^{-1}(x+)$, then we obtain even that $G$, hence $F$, is constant on $[G_l^{-1}(x),G_l^{-1}(x+)]$, making $F\circ G_l^{-1}$ yet right-continuous at $x$, a contradiction.

(2) Consequently, if $F\circ G_l^{-1}$  is not right-continuous at $x$, and hence $G$ has a jump at $G_l^{-1}(x+)$, then $G_l^{-1}$ is locally constant strictly to the right of $x$.

Finally, to see $\Phi$ is continuous, note first that $\Phi$ is certainly continuous at $0$ and $1$:
\begin{itemize}
\item Continuity at $1$. If $1$ is an accumulation point of $\IM G\backslash \{1\}$ then either $\lim_{1-}G_l^{-1}=\infty$, and it follows from $\lim_{\infty }F=1$, and the nondecreasingness of $\Phi$; or else $G_l^{-1}(1-)\in \mathbb{R}$, $G$ is constant equal to $1$ after $G_l^{-1}(1-)$ and does not jump there, so the same holds of $F$, and again continuity follows easily. Otherwise $\Phi$ approaches $1$ linearly, so continuously.
\item Continuity at $0$. If $0$ is not an accumulation point of $\IM G\backslash \{0\}$, then $\Phi$ approaches it linearly, so continuously. Otherwise, if $0$ is an accumulation point of $\IM G\backslash \{0\}$, then either $G_l^{-1}(0+)=-\infty$, whence the claim follows from $\lim_{-\infty}F=0$; or else $G_l^{-1}(0+)\in \mathbb{R}$, $G$, hence $F$, vanishes on $(-\infty,G_l^{-1}(0+))$ and cannot jump at $G_l^{-1}(0+)$, so neither can $F$, and again the claim follows.
\end{itemize}
Next, for any $x\in (0,1)\backslash \IM G$, $\Phi$ is also continuous on $[\ux,\ox]$, thanks to Remark (2) and the nondecreasingness of $\Phi$:
\begin{itemize}
\item Continuity on $(\ux,\ox)$ is trivial.
\item Continuity at $\overline{x}$: we may assume $\ox\ne 1$. Continuity from the right at $\overline{x}$, if $\overline{x}$ is an accumulation point of $\IM G\backslash \{\ox\}$, follows from Remark (2) and the nondecreasingness of $\Phi$. Continuity from the left at $\overline{x}$ is trivial.
\item Continuity at $\ux$: we may assume $\ux\ne 0$. When $\ux\in\IM G$, and $\ux$ is an accumulation point of $\IM G\backslash \{\ux\}$, it follows from the nondecreasingness of $\Phi$, which gives $F\circ G_l^{-1}(\ux-)$ is $\geq$, hence $=$, to $F\circ G_l^{-1}(\ux)$. When $\ux\notin\IM G$, continuity from the right at $\ux$ is trivial, and continuity from the left at $\ux$ also follows trivially since then necessarily $\ux$ is an accumulation point of $\IM G\backslash \{\ux\}$.
\item There remains to show continuity at $x$ of $\Phi$ when $x$ is an isolated point of $\IM G\cap (0,1)$. Continuity from below is clear. Continuity from above follows simply from the nondecreasingness of $\Phi$ that implies $(F\circ G_l^{-1})(x-)$ is $\geq$, hence $=$, to $F(G_l^{-1}(x))$.
\end{itemize}
 Then, if $x\in (0,1)$ is a point of discontinuity of $\Phi$, it is necessarily from $\IM G$; and:
\begin{itemize}
\item Either $\Phi$ is not left-continuous at $x$. Then either $[x-\epsilon,x)\cap \IM G=\emptyset$ for some $\epsilon>0$, or else $F\circ G_l^{-1}$ must not be left-continuous at $x$. If the latter holds, then $F$ has a jump at $G_l^{-1}(x)$, and $x$ must have been a left increase point of $G_l^{-1}$. So $G$ must have had a jump at $G_l^{-1}(x)$ and the left limit of $G$ at $G_l^{-1}(x)$ must have been equal to $x$. But then $G_l^{-1}$ is locally constant strictly to the right of $x$.
\item Or else $\Phi$ is not right-continuous at $x$. Again either $(x,x+\epsilon]\cap \IM G=\emptyset$ for some $\epsilon>0$, or else $F\circ G_l^{-1}$ must not be right-continuous at $x$. Assuming the latter, by Remark (2) above, $G_l^{-1}$ is locally constant strictly to the right of $x$.
\end{itemize}
In every case $x$ belongs to the closure of a non-degenerate interval of constancy of $G_l^{-1}$, which is to say, that $\Phi$ is continuous there, a contradiction.
\end{proof}

\begin{lemma}[``Inverse Cauchy-Schwartz'']\label{lemma:reverse-Cauchy-Schwartz}
Let $\LLL$ be a probability law on $\mathcal{B}(\mathbb{R})$. Then for any $n\in \mathbb{N}_0$, and any choice of nonnegative, nonincreasing, right-continuous (respectively left-continuous and bounded) functions $f_1,\ldots,f_n$ on $\RR$, we have $\int \prod_{i=1}^nf_id\LLL\geq \prod_{i=1}^n\int f_id\LLL$.
\end{lemma}
\begin{proof}
The claim holds true when $f_i=\mathbbm{1}_{(-\infty,a_i)}$ (respectively $f_i=\mathbbm{1}_{(-\infty,a_i]}$) for some $a_i\in \mathbb{R}$, for all $i\in [n]$, since in that case $\LLL(-\infty,\land_{i=1}^na_i)\geq \prod_{i=1}^n\LLL(-\infty,a_i)$ (respectively $\LLL(-\infty,\land_{i=1}^na_i]\geq \prod_{i=1}^n\LLL(-\infty,a_i]$), thanks to $\LLL$ being a probability law. Then, for each $i\in [n]$ separately and successively (keeping each time the $f_j$, $j\ne i$, fixed), we may extend the inequality: first, by nonnegative linearity, to $f_i$ that are simple, right-continuous (respectively left-continuous), non-increasing and nonnegative functions on $\mathbb{R}$; second, by monotone (respectively bounded) convergence, to the general case (approximating $f_i$, e.g. by $f^N_i=\sum_{k\in \ZZ}\mathbbm{1}_{[\frac{(k-1)\land (N2^N)}{2^N},\frac{k\land (N2^N)}{2^N})}f_i(\frac{k\lor (-N2^N)}{2^N})$, respectively $f^N_i=\sum_{k\in \ZZ}\mathbbm{1}_{[\frac{(k-1)\land (N2^N)}{2^N},\frac{k\land (N2^N)}{2^N})}f_i(\frac{(k-1)\lor (-N2^N)}{2^N})$, letting $N\to \infty$ over $\mathbb{N}_0$).
\end{proof}

\end{document}